\documentclass[11pt, reqno]{amsart}
\usepackage{amsrefs, amsmath, amssymb, amsthm, graphicx,marvosym}
\usepackage{cancel}

\def\eq#1{(\ref{#1})}
\def\nn{\nonumber}
\def\({\left(\begin{array}{cccccc}}
\def\){\end{array}\right)}

\def\com#1{\quad{\textrm{#1}}\quad}
\def\eq#1{(\ref{#1})}
\def\nn{\nonumber}

\def\({\left(\begin{array}{cccccc}}
\def\){\end{array}\right)}

\def\bes{\begin{eqnarray}}
\def\ees{\end{eqnarray}}

\newcommand{\ba}{\overset{\raisebox{0pt}[0pt][0pt]{\text{\raisebox{-.5ex}{\scriptsize$\leftharpoonup$}}}}}
\newcommand{\fa}{\overset{\raisebox{0pt}[0pt][0pt]{\text{\raisebox{-.5ex}{\scriptsize$\rightharpoonup$}}}}}

\newcommand{\beq}{\begin{equation}}
\newcommand{\eeq}{\end{equation}}
\newcommand{\bea}{\begin{eqnarray}}
\newcommand{\eea}{\end{eqnarray}}
\newcommand{\beann}{\begin{eqnarray*}}
\newcommand{\eeann}{\end{eqnarray*}}

\newcommand{\ve}{\varepsilon}

\newcommand{\bp}{\begin{proof}}
\newcommand{\ep}{\end{proof}}

\newtheorem{theorem}{Theorem}[section]

\newtheorem{corollary}[theorem]{Corollary}
\newtheorem{lemma}[theorem]{Lemma}
\newtheorem{definition}[theorem]{Definition}

\newtheorem{remark}[theorem]{Remark}

\numberwithin{equation}{section}

\begin{document}

\title[Lower bound of density for isentropic gas]
{Lower bound of density for Lipschitz continuous solutions in the isentropic gas dynamics}
\author{Geng Chen}
\address{Geng Chen, School of Mathematics,
Georgia Institute of Technology, Atlanta, GA 30332 USA ({\tt gchen73@math.gatech.edu})}

\author{Ronghua Pan}
\address{Ronghua Pan, School of Mathematics,
Georgia Institute of Technology, Atlanta, GA 30332 USA ({\tt panrh@math.gatech.edu})}

\author{Shengguo Zhu}
\address{Shengguo Zhu, Shanghai Jiaotong University, Shanghai, China and Georgia Institute of Technology, Atlanta, GA 30332 USA  ({\tt zhushengguo@sjtu.edu.cn})}

\date{\today}

\begin{abstract} For the Euler equations of isentropic gas dynamics in one space dimension, also knowns as p-system in Lagrangian coordinate, it is known that the density can be arbitrarily close to zero as time goes to infinity, even when initial density is uniformly away from zero.  
In this paper, for uniform positive initial density, we prove the density in any Lipschitz continuous solutions for
Cauchy problem has a sharp positive lower bound
in the order of $O(\frac{1}{1+t})$, which is  identified by explicit examples in \cite{courant}. 
\end{abstract}

\maketitle
\bigskip


\section{Introduction}
In this paper we consider the Cauchy problem for isentropic gas dynamics
in the Lagrangian coordinate
	\beq\label{ps}\left\{
		\begin{array}{ll}
		&{v}_t-u_x = 0\,\vspace{.1cm}\\
		&u_t+p({v})_x =0\,,\vspace{.1cm}\\
		&v(x,0)=v_0(x),\quad u(x,0)=u_0(x)\,,
		\end{array}\right.
	\eeq
where  the specific volume ${v}=1/\rho$, the density  $\rho>0$ and the velocity $u$ of the gas
are all functions on $(x,t)\in{\mathbb R}\times{\mathbb R}^+$. The pressure $p({v})$
satisfies
\beq\label{pga}
p({v})=K{v}^{-\gamma} \com{with} \gamma>1\,.
\eeq
This system is also called the p-system. The Lipschitz continuous solution of \eqref{ps} is equivalent to the solution in Eulerian coordinate \cite{wagner}.
Let
\[
c=\sqrt{-{p}'({v})}\,.
\]
be the sound speed. The Riemann invariants $s$ and $r$ are defined as
\beq\label{s_r_eq}
s:=u-\phi \qquad r:=u+\phi
\eeq
with
\beq\label{xi_eq}
 \phi\equiv\phi({v}):=\int^{{v}}_{1} \sqrt{-{p}'({v})}\,d{v}\,.
\eeq
For smooth solutions, $s$ and $r$ satisfy
\[
s_t+cs_x=0\qquad r_t-cr_x=0\,.
\]

Toward a large date theory, such as the existence of BV solutions
for isentropic Euler equations (\ref{ps}), one of the main challenges is the possible
degeneracy when density approaches vacuum.
When the solution approaches vacuum, it causes major difficulties in analyzing the large data solutions for \eqref{ps}, 
because  \eqref{ps} loses its strict hyperbolicity when $\rho=0$. See \cites{BCZ,CJ,ls} for analysis and examples showing 
these difficulties.
Therefore, sharp information on the time decay of density lower bound is critical in the study of compressible Euler 
equations.

It is well known for \eqref{ps} that the density can be arbitrarily close to zero as time goes to infinity, even when initial density is uniformly away from zero, such as in the interaction of two strong rarefaction waves,
c.f. \cites{courant,Riemann}. In fact, the study of interaction between two rarefaction waves
can be found in Riemann's pioneer paper \cite{Riemann} in 1860.  By studying Riemann's 
construction, when $\gamma=\frac{2N+1}{2N-1}$ with any positive integer $N$, Lipschitz 
continuous examples were provided in Section 82 in \cites{courant}, in which density functions are proved to decay to 
zero in an order of $O(1+t)^{-1}$. For reader's convenience, a relative detailed discussion can be found in Section 2.
The main result we show in this paper is that when $1<\gamma<3$, in any Lipschitz 
continuous solutions, density has a sharp positive lower bound in the order of $O(1+t)^{-1}$.

Due to the elegant structure of Euler equations,
local behavior of Lipschitz continuous solution can be classified into two classes: compression and rarefaction,
defined below.
There have been many efforts put on this problem \cites{G3,G6,lax2,youngblake1}.
 \begin{definition}\label{def1}
At any point on $(x,t)$-plane, the smooth solution is forward (resp. backward) rarefaction {\em$\fa R$ ($\ba R$)} if and only if $s_x\geq0$ {\em(}reap. $r_x\geq0${\em)} at that point;
forward {\em(}resp. backward{\em)} compressive {\em$\fa C$ ($\ba C$)} if and only if $s_x<0$ {\em(}reap. $r_x<0${\em)} at that point.
\end{definition}

Among many results, two of them are closely related to current paper. 
For rarefactive piecewise Lipschitz continuous solutions,  Longwei Lin proves that the density has a $O(1+t)^{-1}$ lower bound in \cite{lin2}. For general smooth solutions, in a very recent paper  \cite{CPZ}, we find a $O(1+t)^{-4/(3-\gamma)}$ lower bound when $1<\gamma<3$, using which together with Lax's decomposition in \cite{lax2}, we prove that gradient blowup of $u$ and/or ${v}$ happens in finite time if and only if the initial data are forward or backward compressive somewhere. 
The second result is further extended in \cite{CPZ} to full (nonisentropic) Euler equations.

In this paper,  for general Lipschitz continuous solution of (\ref{ps}) when $1<\gamma<3$, we improve the lower bound on density from
$O((1+t)^{-4/(3-\gamma)})$  to the optimal order of $O((1+t)^{-1})$. Based on this result, we improve the estimate in \cite{CPZ} on the life-span for classical solution of \eqref{ps} including compression.

\bigskip

Our main theorem is

\begin{theorem}{\label{main0}}
Assume that initial data $s_0(x)=s(x,0)$ and $r_0(x)=r(x,0)$ are Lipschitz continuous
functions on $x$.
Furthermore, assume that ${v}(x,0)$ in the initial data has uniformly positive upper and lower bounds.
Suppose that  $\big(u(x,t),{v}(x,t)\big)$ is a Lipschitz continuous weak solution for the initial value problem of (\ref{ps}) with $1<\gamma<3$,
when $(x,t)\in{\mathbb R}\times [0,T]$, where $T$ 
can be any positive number. 
Then
\beq\label{main0_esti_0}
{v}(x,t)\leq \max_x\big({{v}(x,0})\big)+L\, t, \com{for any} (x,t)\in{\mathbb R}\times [0,T]
\eeq
with some positive constant $L$ depending only on the initial data.
\end{theorem}

To prove this Theorem, we study a polygonal scheme similar to the one used in \cite{lin2,lxy}.
The polygonal scheme was first established by Dafermos in \cite{dafermos2} in the  study of scalar conservation law
then was modified by Diperna for system of conservation laws. The scheme has been widely used for the well-posedness and behaviors of hyperbolic conservation laws \cite{bressan, Dafermos2010,TC}.

In the polygonal scheme, we divide the $(x,t)$-plane into finite districts, on each of
which the forward (resp. backward) waves are in the same type: forward (resp. backward) rarefaction or compression. 

Because the density increases when it crosses a compressive wave, it seems that only districts including forward and backward rarefaction waves  directly make the density decreasing. However, things are more complicated than this. In fact,
for rarefaction-rarefaction districts adjacent to initial line, we can directly use a similar argument in \cite{lin2} to find the desired bound on density, where the bound depends on $\max_{x\neq y}\frac{s_0(x)-s_0(y)}{x-y}$ and $\max_{x\neq y}\frac{r_0(x)-r_0(y)}{x-y}$, where these two one-side Lipschitz constants mean how rarefactive the initial data are. However, the major difficulty we conquer in this paper is how to analyze those rarefaction-rarefaction districts far away from the initial line. The rarefaction waves in these districts have passed some compressions in the opposite families before reaching the rarefaction-rarefaction districts. So $\max_{x\neq y}\frac{s(x,t)-s(y,t)}{x-y}$ and $\max_{x\neq y}\frac{r(x,t)-r(y,t)}{x-y}$ might increase on time, hence to obtain a sharp lower bound on density,
we need to carefully analyze all three types of districts: rarefaction-rarefaction, rarefaction-compression and compression-compression districts.

The key new idea is given in  Lemma \ref{a_t_in}, in which we show that a function $a(t)$ defined in Definition \ref{def_at} is not increasing on time, in the scheme approximating general smooth solutions. Using this lemma, we could piece up the density estimates we got in each district into a global one up to time $T$ on the district of Lipschitz continuous solutions.

As a corollary of Theorem \ref{main0}, we achieve a better estimate for the life-span of classical solutions including compression than \cite{CPZ}, when $1<\gamma<3$. 

This paper is organized in the following order. In Section 2, we review the explicit example given in \cite{courant} in which density approaches zero in an order of $O(1+t)^{-1}$. In Section 3,
we review the polygonal scheme and define Rarefaction/Compression character. In Section 4,
we prove the main Theorem on lower bound of density.

%
\section{Exact interaction between two rarefactions}
In this section, we review the concrete example for interaction between two centered rarefaction simple waves provided in Section 82 in \cite{courant}. Please find detail calculations in Section 82 in \cite{courant}.

	\begin{figure}[htp] \centering
		\includegraphics[width=.4\textwidth]{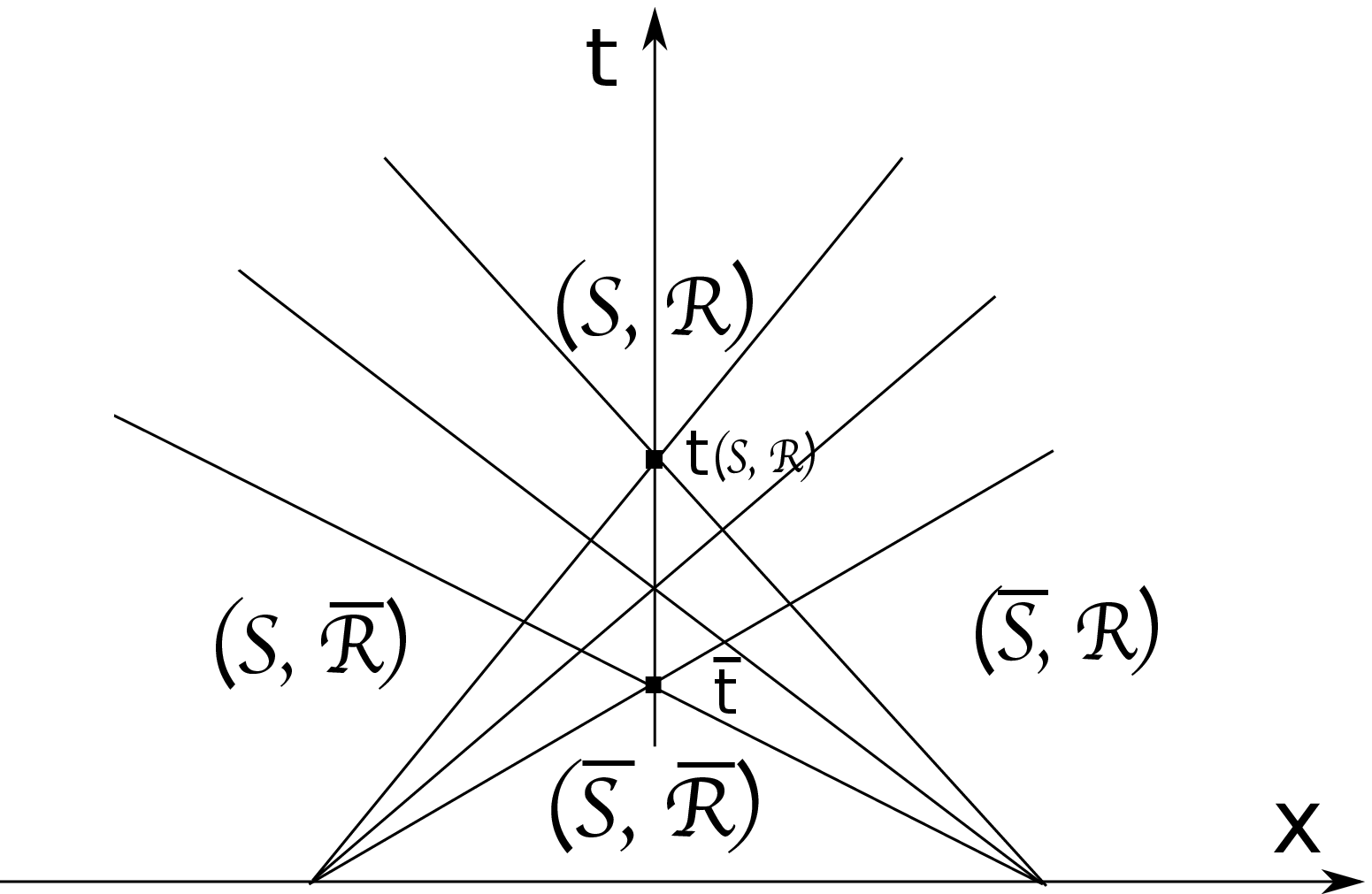}\qquad
		\includegraphics[width=.4\textwidth]{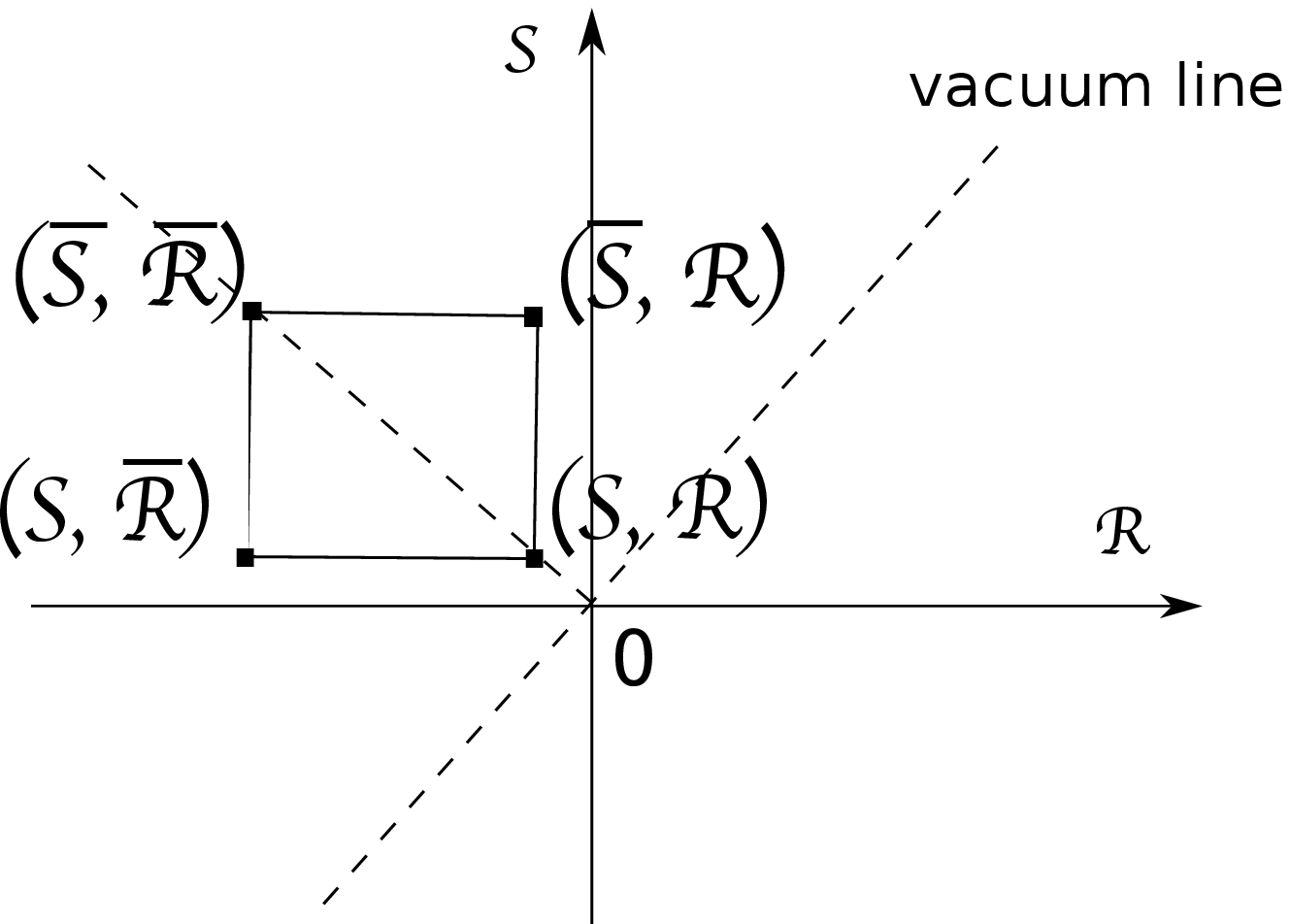}
		\caption{\label{cournat_inter}Interaction of two centered rarefaction waves.}
	\end{figure}
	
	In this section, we use slightly different Riemann invariants $\mathcal S$
	and $\mathcal R$, instead of $s$ and $r$. We denote
\beq\label{math_S_R}
\mathcal S=u+\textstyle\frac{2\sqrt{K}}{\gamma-1}v^{\frac{1-\gamma}{2}},\quad
\mathcal R=u-\textstyle\frac{2\sqrt{K}}{\gamma-1}v^{\frac{1-\gamma}{2}}.
\eeq
In fact, $\mathcal{S}$ and $s$
	and $\mathcal{R}$ and $r$ are different by two constants, respectively. It is easy to get
	\beq\label{math_S_R2}
\mathcal S-\mathcal R=\textstyle\frac{4\sqrt{K}}{\gamma-1}v^{\frac{1-\gamma}{2}}\,.
\eeq

For simplicity, in Figure \ref{cournat_inter}, we only consider an interaction between two centered rarefaction waves,
where we assume that the first interaction happens at $t=\bar t>0$ and  $x=0$, with constant state 
$(\bar{\mathcal{S}}, \bar{\mathcal{R}})$ below the point $(0,\bar t)$. Furthermore, we assume that $u(0,\bar t)=0$,
hence
\[
\bar{\mathcal{S}}=-\bar{\mathcal{R}}>0.
\]
For simplicity, we assume that $\mathcal{S}$ in the left state of interaction is always positive and 
$\mathcal{R}$ in the right state of interaction is always negative. 

Then using the fact that $\mathcal{S}$ and $\mathcal{R}$ are constant along forward and backward characteristics, respectively, in \cite{Riemann}, Riemann first found below equation:
\beq\label{t_S_R}
t(\mathcal{S},\mathcal{R})=\bar t\left(\frac{\bar{\mathcal{S}}-\bar{\mathcal{R}}}
{\mathcal{S}-{\mathcal{R}}}\right)^\alpha
F\left( 1-\alpha,\alpha,1,\frac{(\bar{\mathcal{S}}-\mathcal{S})(\bar{\mathcal{R}}-\mathcal{R})}{(\bar{\mathcal{S}}-\bar{\mathcal{R}})(\mathcal{S}-\mathcal{R})} \right)
\eeq
where $t(\mathcal{S},\mathcal{R})$ is the time when the interaction ends, $F(z_1,z_2,z_3,z_4)$ is a hypergeometric function and
\beq\label{la_def}
\alpha=\frac{\gamma+1}{2(\gamma-1)}.
\eeq
When $\alpha$ is a positive integer $N$ or equivalently 
\[
\gamma=\frac{2N+1}{2N-1}\,.
\]
Then \eqref{t_S_R} can be simplified into
\beq\label{t_S_R0}
t(\mathcal{S},\mathcal{R})=\bar t\left(\frac{\bar{\mathcal{S}}-\bar{\mathcal{R}}}
{\mathcal{S}-{\mathcal{R}}}\right)^\alpha 
P_{\alpha-1}
\left(\frac{1}{\bar{\mathcal{S}}}\frac{\bar{\mathcal{S}}^2-\mathcal{S}\mathcal{R}}{\mathcal{S}-\mathcal{R}}\right)
\eeq
where $P_{\alpha-1}(z)$ is the Legendre's function, which is a $(\alpha-1)$'s order polynomial.

Recall that we assume that $\mathcal{S}$ in the left state of interaction is always positive and 
$\mathcal{R}$ in the right state of interaction is always negative. Let  $|\mathcal S|$ and $|\mathcal R|$ be both very small, 
then density $\rho$ at the point where interaction ends (at time $ t(\mathcal{S},\mathcal{R})$) is very close to zero. Using (\ref{math_S_R2}), to the leading order, we have 
 \[
 t(\mathcal{S},\mathcal{R})=O(\rho^{-1})
 \]
 where we used (\ref{math_S_R}),  \eqref{la_def}, \eqref{t_S_R0},
 $P_{\alpha-1}(z)$ is a $(\alpha-1)$'s order polynomial and 
 $\frac{1}{\bar{\mathcal{S}}}(\bar{\mathcal{S}}^2-\mathcal{S}\mathcal{R})>\bar{\mathcal{S}}$ is uniformly positive. 
When $|\mathcal S|$ and $|\mathcal R|$ both approach zero, 
 \[
  t(\mathcal{S},\mathcal{R})\rightarrow \infty\,.
 \]
   So when $t$ is large enough,
  \[
 \min_x\rho(x,t)=O(t^{-1})\,.
 \]

It is clear that the initial density  in the explicit example is uniformly positive,
by the right picture of Figure \ref{cournat_inter}.

%
\section{The polygonal scheme and Rarefaction/Compression character}
In this section, we first review some basic setup for the polygonal scheme, following the notations in \cites{lin2,lxy}.
%
\subsection{Pressure, Riemann invariants and Standard states\label{subsec_2.1}}
The polygonal approximation of $p({v})$ is defined as follows.

For any given positive integer $n$, let ${v}_0^{(n)}=1$ and ${v}_k^{(n)}$ with integer $k$
determined by the recurrence formula
\beq\label{G_def}
G({v}_k^{(n)},{v}_{k+1}^{(n)}):=\Big(p({v}_{k}^{(n)})-p({v}_{k+1}^{(n)}) \Big)\,\Big({v}_{k+1}^{(n)}-{v}_{k}^{(n)}\Big)=\frac{1}{n^2}\,.
\eeq

It is easy to check that for each fixed $n$ there exists a unique sequence $\left\{ {v}_{k}^{(n)}\right\}$ with positive integer $k$, defined by \eqref{G_def}, such that 
\[
\lim_{k\rightarrow \infty}{v}_{k}^{(n)}=\infty,\quad\lim_{k\rightarrow -\infty}{v}_{k}^{(n)}=0\,.
\]
Furthermore, denote
\[
\delta_k^{(n)}:={v}_{k+1}^{(n)}-{v}_k^{(n)}\,,
\]
hence,
\[
\lim_{k\rightarrow \infty}\delta_k^{(n)}=\infty\,.
\]
The polygonal lines with vertices $\bigl({v}_k^{(n)},\, p({v}_k^{(n)})\bigr)$
are the polygonal approximation of $p({v})$, denoted by $p^{(n)}({v})$.

Define
\beq\label{Phi_def}
\Phi^{(n)}({v}):=\int^{{v}}_1 \sqrt{-{p^{(n)}}'({v})}\,d{v}\,,
\eeq 
then
\[
\Phi^{(n)}({v}_{0}^{(n)})=\Phi^{(n)}(1)=0\,,
\]
and
\[
\Phi^{(n)}({v}_{k+1}^{(n)})-\Phi^{(n)}({v}_{k}^{(n)})=\sqrt{G({v}_k^{(n)},{v}_{k+1}^{(n)})}=\frac{1}{n}\,,
\]
therefore
\[
\Phi^{(n)}({v}_{k}^{(n)})=\frac{k}{n}\,,
\]
where $k$ is an integer and $n$ is a positive integer.

%
We define
\beq\label{r_s_def}
r^{(n)}(u,{v})=u+\Phi^{(n)}({v}),\qquad s^{(n)}(u,{v})=u-\Phi^{(n)}({v})\,,
\eeq
which are corresponding to the Riemann invariants $r$ and $s$ defined in \eqref{s_r_eq}, respectively. 

The following states are called standard states:
\[
\big(u,\,{v}\big)=\Big(\frac{i}{n},\,{v}^{(n)}_j\Big), \quad i.e.\quad 
\big(u,\,\Phi^{(n)}\big)=\Big(\frac{i}{n},\,\frac{j}{n}\Big)\,,
\]
and
\[
\big(r^{(n)},\, s^{(n)}\big)=\Big(\frac{2k}{n},\,\frac{2l}{n}\Big)\,,
\]
where $i$ and $j$ are integers and
\[
k=\frac{1}{2}(i+j)\qquad l=\frac{1}{2}(i-j)\,.
\]

For convenience, we might omit the superscript $(n)$ if there are no confusions.
%
\subsection{Riemann problems\label{sub_sec_RP}}
Then we consider the following Riemann problem:edge

	\beq\label{ps-l}\left\{
		\begin{array}{rcl}
		{v}_t-u_x &=& 0\,\vspace{.2cm}\\
		u_t+p^{(n)}({v})_x &=&0\,,\vspace{.2cm}\\
		\big(u_0(x),\, {v}_0(x)\big)
		&=&\left\{
  \begin{array}{ll}
  \big(u_-,\, {v}_-\big), & x<0\,,\vspace{.2cm}\\
  \big(u_+,\, {v}_+\big), & x>0\,,\\
  \end{array}
  \right.	
		\end{array}\right.
	\eeq
with
\[
  \big(u_-, \,{v}_-\big)=\Big(\frac{i}{n},\,{v}^{(n)}_j\Big)\,, \qquad
   \big(u_+, \,{v}_+\big)=\Big(\frac{i+M+N}{n},\,{v}^{(n)}_{j+M-N}\Big) \,,
\]
where $M$ (and $N$) can be $-1$, $0$ or $1$.
Clearly, we have
\[
\big(r^{(n)}_0(x),\, s^{(n)}_0(x)\big)
=\left\{
  \begin{array}{ll}
  \big(r^{(n)}_-,\, s^{(n)}_-\big), & x<0\,,\vspace{.2cm}\\
  \big(r^{(n)}_+,\, s^{(n)}_+\big), & x>0\,,\\
  \end{array}
  \right.			
\]
with
\[
\big(r^{(n)}_-,\, s^{(n)}_-\big)=\Big(\frac{2k}{n},\,\frac{2l}{n}\Big)\,,\qquad
\big(r^{(n)}_+,\, s^{(n)}_+\big)=\Big(\frac{2(k+M)}{n},\,\frac{2(l+N)}{n}\Big)\,.
\]

The solution of the Riemann problem of  \eqref{ps-l} consists of $|M|+|N|+1$ standard states,
divided by $|M|+|N|$ jump discontinuities (straight lines centered at the origin). 

To calculate the middle state $(r^{(n)}_m,s^{(n)}_m)$ in the solution of Riemann problem, we use
following criterions:jump

\bes
\label{s_r_con_approx}
s^{(n)} \com{and} r^{(n)}  \com{are constants across backward and forward jumps, respectively.} \\
\nn
\ees
This criterion is corresponding to \eqref{s_r_eq} for the smooth solution.
Hence the middle state in the solution of Riemann problem is always 
\[
\big (r^{(n)}_m,s^{(n)}_m\big)=\Big(\frac{2(k+M)}{n},\frac{2l}{n}\Big).
\]
Then it is easy to have 
\[\big(u_m, {v}_m\big)=\Big(\frac{i+M}{n},\,{v}^{(n)}_{j+M}\Big).\]

\begin{remark}
We note that
\begin{itemize}
\item
if $M=1$ (resp. $N=1$),
the backward (resp. forward) jump discontinuity describes a rarefactive wave.
\item
If $M=0$ (resp. $N=0$),
there are no backward (resp. forward) jump discontinuity.
\item
if $M=-1$ (resp. $N=-1$),
the backward (resp. forward) jump discontinuity describes a compressive wave.
\end{itemize}
The definition of rarefaction and compression are in Definition \ref{def1}. We refer the reader to \cites{BCZ, CJ} for more details on wave curves for rarefaction and compression waves.
\end{remark}

When $M$ is $-1$ or $1$, the slope of the backward jump is
\beq \label{geng2}\begin{split}
\ba \lambda=-\sqrt{-\frac{p({v}^{(n)}_{j+M})-p({v}^{(n)}_{j})}{{v}^{(n)}_{j+M}-{v}^{(n)}_{j}}}=\frac{-1}{n\,
\big|{v}^{(n)}_{j+M}-{v}^{(n)}_{j}\big|}<0.
\end{split}
\eeq
When $N$ is $-1$ or $1$, the slope of the forward jump is
\beq \label{geng3}\begin{split}
\fa \lambda=\sqrt{-\frac{p({v}^{(n)}_{j+M})-p({v}^{(n)}_{j+M-N})}{{v}^{(n)}_{j+M}-{v}^{(n)}_{j+M-N}}}=\frac{1}{n\,\big|{v}^{(n)}_{j+M}-{v}^{(n)}_{j+M-N}\big|}>0.
\end{split}
\eeq

%

%
%
\subsection{The polygonal scheme and Rarefactive/Compressive characters}
For any given Lipschitz continuous initial data $(r_0,{s}_0)$ with $u_0 $ and ${v}_0$ uniformly bounded and ${v}_0$ uniformly away from zero, similar as in \cite{lin2}, we can find a sequence of piecewise constant functions $(u^{(n)}_0,{v}^{(n)}_0)$ which takes values on finitely many standard states. More precisely, 
\[\big(r^{(n)}(x),\, s^{(n)}(x)\big)=\Big(\frac{2k^{(n)}_{\alpha}}{n},\,\frac{2l^{(n)}_{\alpha}}{n}\Big)\,,\] 
for some integers $k^{(n)}_{\alpha}$ and $l^{(n)}_{\alpha}$ as $x\in(x^{(n)}_{\alpha}, x^{(n)}_{\alpha+1})$,
where integer $\alpha$ is from $1$ to $j$ and $x^{(n)}_1=-\infty$ and $x^{(n)}_{j+1}=\infty$.
Furthermore, $(k^{(n)}_{\alpha-1}, l^{(n)}_{\alpha-1})$
is different from $(k^{(n)}_{\alpha}, l^{(n)}_{\alpha})$ and 
\[
|k^{(n)}_{\alpha-1}-k^{(n)}_{\alpha}|\leq1,\qquad |l^{(n)}_{\alpha-1}-l^{(n)}_{\alpha}|\leq1\,.
\]

We have $(u^{(n)}_0,{v}^{(n)}_0)\rightarrow(u_0,{v}_0)$ uniformly, and
\beq\label{J_est}
\max_{\alpha}{\frac{2}{n\Big(x^{(n)}_{\alpha+1}- x^{(n)}_{\alpha}\Big)}}\longrightarrow\max_{x\neq y}\Big\{\frac{s_0(x)-s_0(y)}{x-y},\frac{r_0(x)-r_0(y)}{x-y}\Big\}, \quad \text{as}\quad n\rightarrow\infty\,.
\eeq

For any positive integer $n$, we solve the Riemann problem at each discontinuity. Note the left, right and middle sates in the solution of each Riemann problem are still
standard states, and split by jump discontinuities. Let these jump discontinuities evolve. When two jumps in different families interact with each other, new Riemann problem appears, which can also be solved. Finally, we get a well-defined polygonal scheme, including finitely many jumps, before the possible interactions between jumps of the same family.

In \cite{lxy}, the authors show that under some regularity condition on the initial data  $(u_0,{v}_0)$, the approximation solutions $(u^{(n)},{v}^{(n)})$ in the polygonal scheme are well-defined, i.e. there is no interaction between jumps of the same family, 
in a time interval $t\in[0,T]$ with $T>0$ only dependent on the $C^1$-norm of $(u_0,{v}_0)$ but independent of $n$. Furthermore, the approximation solutions converge to a classical solution for p-system when $n\rightarrow\infty$. We will give more details on this local-in-time convergence result later.

In this subsection, we first assume that there is no interaction between jumps of the same family. 

\begin{definition}
In the polygonal scheme, for any positive integer $n$, the $(x,t)$-plane is divided into finite {\em blocks} by finitely many jump discontinuities. If a block is a diamond, we call it a {\em diamond block} or {\em diamond}.
Each jump discontinuity is also divided into finite pieces, which are denoted by {\em jump edges}, by finite many intersection points between jumps. 
\end{definition}
To be precise on the definitions of a block or a jump edge, we note $s$ and $r$ are both constant inside each block and on each side of a jump edge, or in another word, there are no other jump discontinuities go inside a block or a jump edge.
\bigskip

To obtain a lower bound on density, it is crucial to study the variation of a jump edge inside
a characteristic tube, such as the propagation of edge $l_1$ in the forward characteristic direction in Figure \ref{R-def12}.

First, we define the Rarefactive/Compressive (R/C) character on a
jump edge. 
\begin{definition}[R/C character on a jump edge]\label{def_RC}
We classify the jump edges into four types:
$R_r$, $R_c$, $C_r$ and $C_c$, where the capital letter denotes the character on
the boundary behind the edge and the subscript denotes 
character on the boundary ahead of the edge. Most of time, we add an arrow 
to denote forward or backward character, respectively.

More precisely, A backward (reap.  forward) jump edge $l_1$ (resp. $l_2$) in the scheme, shown in Figure \ref{R-def12}, is said to be:
\begin{itemize}
\item[i.] $\fa R_r$ (resp. $\ba R_r$), if $u_-< u_+<u_{++}$ (resp. $u_{--}< u_-<u_+$). 
\item[ii.] $\fa R_c$ (resp. $\ba R_c$), if $u_-< u_+$ and $u_+>u_{++}$  (resp. $u_-< u_+$  and $u_{--}> u_-$).
\item[iii.] $\fa C_r$ (resp. $\ba C_r$), if $u_-> u_+$ and $u_+<u_{++}$  (resp. $u_-> u_+$  and $u_{--}< u_-$). 
\item[iv.] $\fa C_c$ (resp. $\ba C_c$), if $u_- > u_+>u_{++}$  (resp. $u_{--}> u_->u_+$).
\end{itemize}
For the left-most forward (resp. right-most backward) jump edge which is unbounded from its left (resp. right) hand side, 
we always say this jump edge belongs to either $C_c$ or $R_r$ by checking the relation of $u$
from the right (resp. left) boundary following above table. 

For simplicity, we always use $C$ to denote $C_r$ or $C_c$ character.
\end{definition}

	\begin{figure}[htp] \centering
		\includegraphics[width=.37\textwidth]{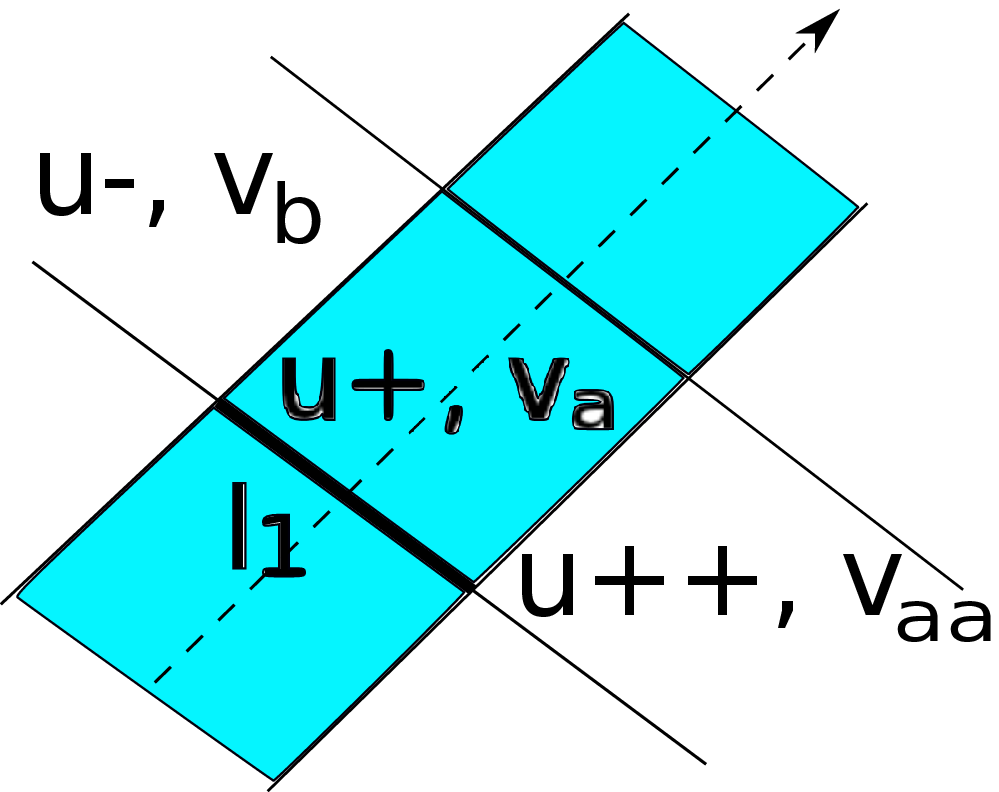}\qquad\qquad	\includegraphics[width=.37\textwidth]{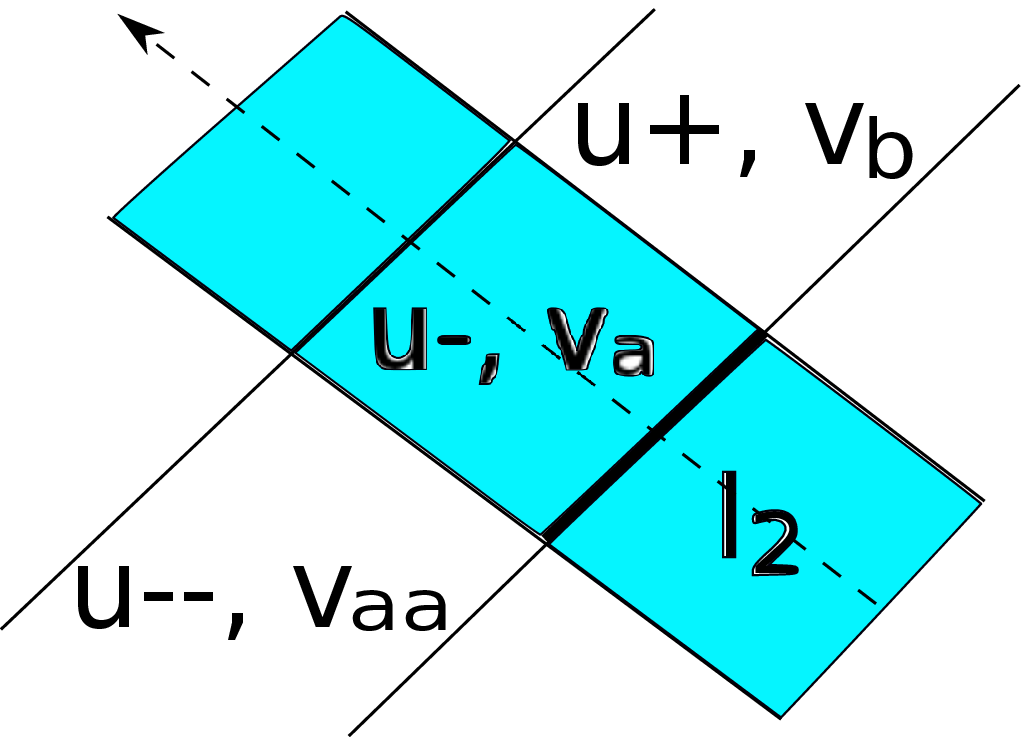}
		\caption{\label{R-def12}Definition of forward (resp. backward) R/C characters for the jump edge $l_1$ (resp. $l_2$).}
	\end{figure}

\begin{remark}\label{last_v}
In Figure \ref{R-def12}, we use the subscripts $a$ and $b$ to denote states ahead of and behind a jump wave front. By (\ref{s_r_con_approx}),
one always has ${v}_b> {v}_a$ for a $R_r$ or $R_c$  jump edge;
and ${v}_b<{v}_a$ for a $C_r$ or $C_c$ jump edge.

In the $R_c$ and $C_r$ pieces, ${v}_{aa}={v}_b$.
\end{remark}

Then we define the $R/C$ character on any blocks.

\begin{definition}[R/C character in a block]
A block is called a $\fa R_r\ba R_r$ block
if its South-West and South-East boundaries are $\fa R_r$ and $\ba R_r$, respectively. 

Similar definitions are also for $\fa R_r\ba R_c$ $\fa R_c\ba R_r$, $\fa R_c\ba R_c$, $\fa C \ba R_r$, $\fa C \ba R_c$, $\fa R_r \ba C$, $\fa R_c\ba C$ and $\fa C \ba C$ blocks.
\end{definition}

	\begin{figure}[htp] \centering
		\includegraphics[width=.3\textwidth]{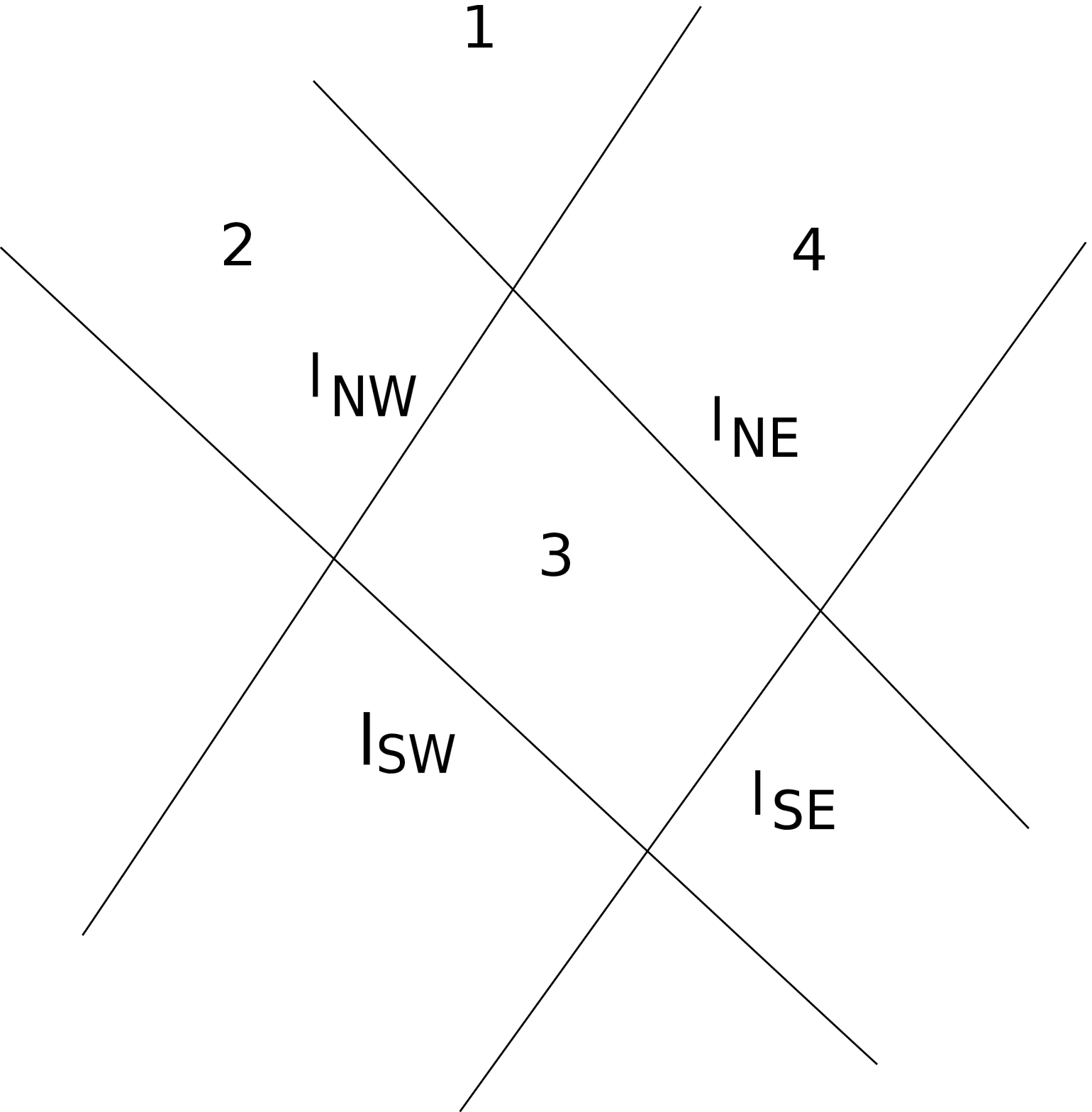}
		\caption{\label{fig_dnc}Proof of Lemma \ref{RC-dnc}.}
	\end{figure}

\begin{lemma}\label{RC-dnc}
The forward (resp. backward) jump edges $l_{SW}$ and $l_{NE}$ (resp.  $l_{SE}$ and $l_{NW}$) shown in Figure \ref{fig_dnc} are same type of jump. 
\end{lemma}
\begin{proof}
By \eqref{r_s_def} and (\ref{s_r_con_approx}), we have
\[
u^{(n)}_1-u^{(n)}_4=\Phi^{(n)}_4-\Phi^{(n)}_1 \quad\text{and}\quad 
u^{(n)}_3-u^{(n)}_2=\Phi^{(n)}_2-\Phi^{(n)}_3 \,,
\]
and
\[
u^{(n)}_1-u^{(n)}_2=\Phi^{(n)}_1-\Phi^{(n)}_2 \quad\text{and}\quad
u^{(n)}_3-u^{(n)}_4=\Phi^{(n)}_3-\Phi^{(n)}_4.
\]
Summing up these equations we show 
\beq\label{lemma_RC-dnc}
u^{(n)}_1-u^{(n)}_2=u^{(n)}_4-u^{(n)}_3\quad\text{and}\quad u^{(n)}_1-u^{(n)}_4=u^{(n)}_2-u^{(n)}_3,
\eeq  
with states $1$$\sim$$4$ given in Figure  \ref{fig_dnc},
hence the monotonicity of $u$ is preserved in both forward and backward directions, which is enough to prove this
lemma by Definition \ref{def_RC}. 
\end{proof}

\begin{definition}[R/C character in a district]
Given a block in certain type  (for example a $\fa R_r\ba R_r$ block),
we define  a \emph{district} in that type (for example a $\fa R_r\ba R_r$ district) to be the largest connected set which includes the given block and consists of only blocks in the same type.
\end{definition}

Finally,  by Remark \ref{last_v} we prove a lemma
showing the decay of ${v}$ in some direction in any diamond or district which is not $\fa R_r\ba R_r$, which will help us find the upper bound on ${v}$ in the next section.

	
\begin{lemma}\label{No-R}
If the forward character of a district $D$ is $\fa C$ or $\fa R_c$,
then $v$ on any block adjacent and above North-west boundary of D is not larger than the maximum $v$ values on
blocks adjacent and below South-East boundary of D. 

And all $\Phi^{(n)}$ values on blocks in $D$ are at most $1/n$ larger than the
maximum $\Phi^{(n)}$ value  on
blocks adjacent and below South-East boundary of D.

Symmetric decay of $v$ happens if the backward character of a district $D$ is $\ba C$ or $\ba R_c$.
\end{lemma}

This lemma is corresponding to a fact in smooth solution that density increases along a characteristic which is  passing through a compressive wave of the other characteristic family, respectively.

%
\section{The lower bound on density in the scheme}
The key idea in the proof of Theorem \ref{main0} is to define a function $a^{(n)}(t)$ for any $n$, which is not decreasing on $t$. 
The monotonicity of $a^{(n)}(t)$ will finally lead to a lower bound on density, 
under the help of the local convergence theorem for the polygonal scheme in \cite{lxy}.
In this section, we always assume that jumps in the same family do not interact. 

To define $a^{(n)}(0)$,
in the first step, we modify some blocks adjacent to initial line $t=0$ to diamonds, as in Figure \ref{scheme5} and left picture of Figure \ref{scheme}. After modification, we call all interior diamonds and boundary diamonds as complete diamonds, also shown in Figure \ref{scheme5}.

	\begin{figure}[htp] \centering
		\includegraphics[width=.6\textwidth]{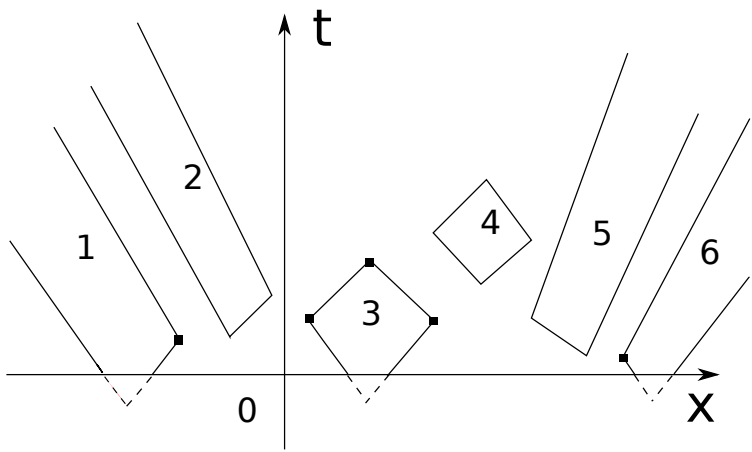}
		\caption{\label{scheme5}
		Modify blocks into diamonds:  interior block 3 (pentagon whose lower boundary is on initial line);  boundary block 1 or 6 which includes one intersection point between jump edges. 
		After modifications, diamonds 1$\sim$6 are all called complete diamonds.
		}
	\end{figure}

\begin{definition}\label{a_0_def}	
Call the collection of all complete diamonds (after modification) as $\mathcal{CD}$. The lowest boundary of 
$\mathcal{CD}$ is a polygonal line, denoted by $t=L^{(n)}_0(x)$. It is clear that $t=L^{(n)}_0(x)$
consists of finitely many jump edges. See Figure \ref{scheme}.

A jumpy edge $JE_i$ on  $t=L^{(n)}_0(x)$ is said in $L^{(n)}_{0,R}$ if it is
$\fa R_r$ or $\ba R_r$. We define the length of the propagation of $JE_i$ onto $x$-axis as
  $a^{(n)}_i(0)$. Now we define
  \[
  a^{(n)}(0)=\min_{JE_i\in L^{(n)}_{0,R}}a^{(n)}_i(0)\,.
  \]
  
  For $T>0$, it is not necessary to modify the diamond again. We thus define $ a^{(n)}(T)$ in a similar way.
\end{definition}

	\begin{figure}[htp] \centering
		\includegraphics[width=.4\textwidth]{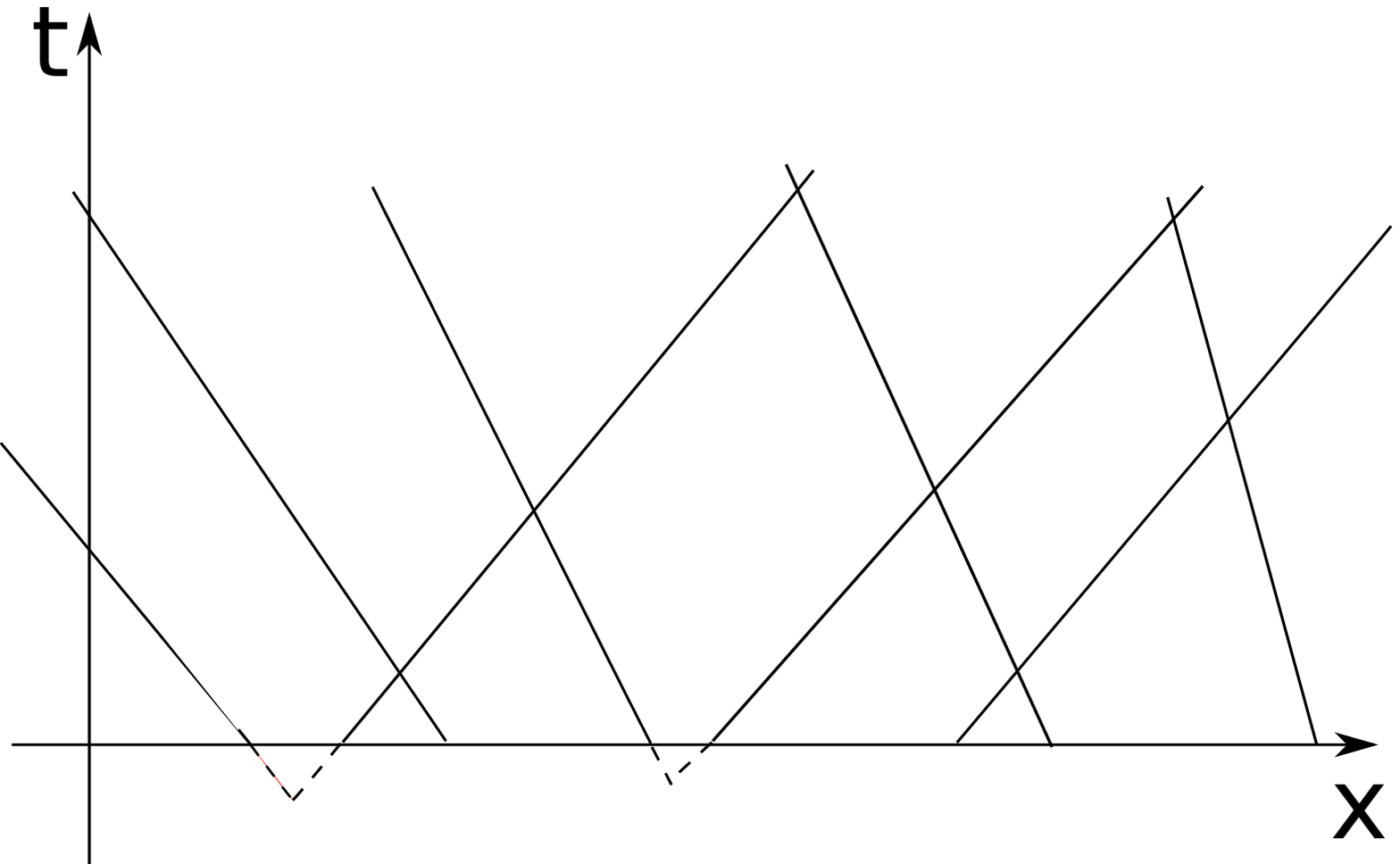}
		\qquad	\includegraphics[width=.4\textwidth]{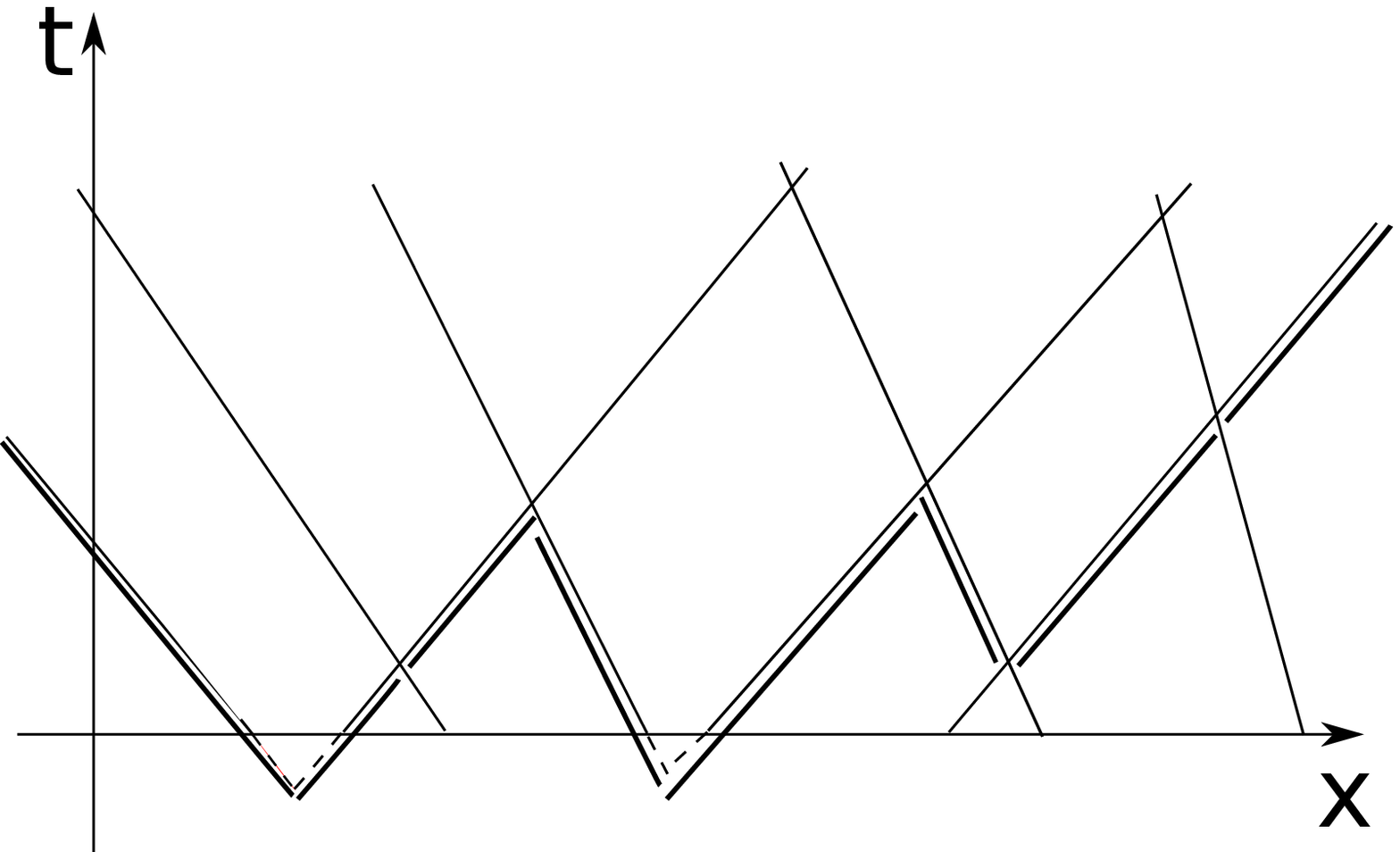}
		\caption{\label{scheme}Left: Modify blocks into diamonds; Right: the definition of $t=L^{(n)}_0(x)$.}
	\end{figure}

\begin{lemma} \label{lemma_a_0}
Assume the initial density $\rho^{(n)}(x,0)$ has positive upper and lower bounds, then
the density $\rho^{(n)}\big(x,L^{(n)}_0(x)+\big)$ has positive lower and upper bounds. And
\beq\label{a_0_M_0}
\lim_{n\rightarrow\infty}\frac{1}{n\, a^{(n)}(0)}\leq M_0 J\,,
\eeq
where
\beq\label{Jdef}
J=\max_{x\neq y}\Big\{\frac{s_0(x)-s_0(y)}{x-y},\frac{r_0(x)-r_0(y)}{x-y}\Big\}
\eeq
 $M_0$ is a positive constant depending on the uniform lower bound on initial density, but $M_0$ is  independent of $n$. 
\end{lemma}

	\begin{figure}[htp] \centering
	\includegraphics[width=.5\textwidth]{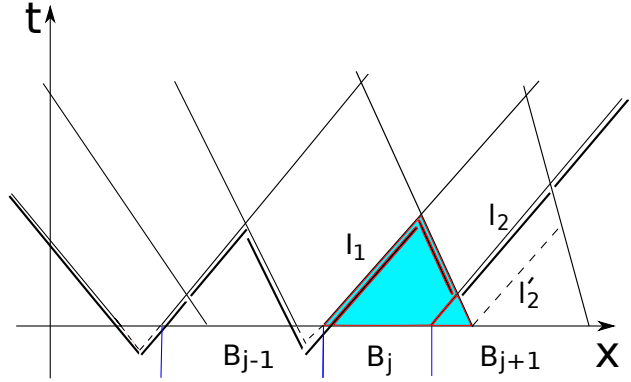}
		\caption{\label{scheme4}Proof of Lemma \ref{lemma_a_0}. $l_2$ and $l'_2$ are parallel with each other.}
	\end{figure}
\begin{proof}
The first claim is clearly true because each state below the curve $t=L_0(x)$ is an initial state in the scheme.

To prove \eqref{a_0_M_0},  it is enough to show that for any jump edge $PQ$ with endpoints $P$ and $Q$ on the curve $t=L^{(n)}_0(x)$, we have 
\beq\label{le_p_1}
\frac{1}{n|x_P-x_Q|}\leq M_0 J+\ve_n
\eeq
where $\ve_n$ goes to zero as $n$ goes to infinity.

We divide jump edges on the curve $t=L^{(n)}_0(x)$ into two types: 
jump edge intersecting with the initial line and jump edge not intersecting with the initial line. 
Without loss of generality, we only consider two forward jump edges $l_1$ and $l_2$ in Figure 
\ref{scheme4}. 

We denote three line segments on $t=0$ divided by adjacent backward jumps as $B_{j-1}$, $B_{j}$ and
$B_{j+1}$, respectively, which are shown in Figure \ref{scheme4}. 

By \eqref{s_r_con_approx}, $l_1$  is ${\ba R}_r$, Definition \ref{def_RC} and the 
discussion in Subsection \ref{subsec_2.1}, we have $r=\frac{2(k-1)}{n}$, $\frac{2k}{n}$ and 
$\frac{2(k+1)}{n}$ in backward tubes, which mean all connected blocks divided by two adjacent backward jumps, including $B_{j-1}$, $B_{j}$ and
$B_{j+1}$, respectively, where $k=\frac{1}{2}(i+j)$ for some integers $i$ and $j$.

By studying the shaded triangle in Figure
\ref{scheme4}, it is easy to get that  there exists a constant $M$ only depending on the upper and lower bounds of initial density, 
such that
\[M\geq\frac{|B_j|}{\Delta_{l_1} x}\]
as $n$ large enough, where  $\Delta_{l_1}x$ means the length of projection of $l_1$ onto the line $t=0$. 
Then we have 
\beq\label{le_p_2}
\frac{\frac{2}{n}}{\Delta_{l_1} x}\leq M\frac{r_{B_{j+1}}-r_{B_{j-1}}}{|B_j|}\,.
\eeq
 Hence \eqref{le_p_1} is clearly correct, because the right hand side of \eqref{le_p_2} is bounded above by $J$ as $n$ goes to infinity, where we use \eqref{J_est}. 

For $l_2$, we note that $\Delta_{l_2}x\geq\Delta_{l'_2}x$ where $l'_2$ is parallel to $l_2$, 
because $l_2$ is ${\ba R}_i$ hence two backward jumps enclosing  $l_2$ opens up. 
A more detailed argument on this fact can be found in Step 2.1 in the proof of Lemma \ref{a_t_in}.  Now we change the problem to a problem for $l'_2$ which intersects with $t=0$, hence similar
as the case for $l_1$, we can prove  \eqref{le_p_1}.

We complete the proof of the lemma.
\end{proof}

The following lemma plays a key role in this paper.
\begin{lemma}\label{a_t_in}
Suppose the scheme is well-defined, then $a^{(n)}(t)$ is not decreasing on $t$.
\end{lemma}
\begin{proof}
We divide our proof into several steps. For convenience, we omit the subscript $(n)$ in the proof
of this lemma.
\paragraph{\bf Step 1}
Suppose jump edges $P_u P_l$ and  $P_{u'} P_l$
are two lower boundaries of a diamond $\Omega$, where 
$P_u$, $P_{u'}$ and $P_l$ are three vertexes of $\Omega$. And $P_l$ is 
the lowest vertex of the diamond.

Then we prove that: for any time $T\geq0$, $P_u P_l$ and  $P_{u'} P_l$ are either {\bf both in} or {\bf both not in}  the 
collection of selected diamonds at time $T$.



Actually, if $P_l$ is in the region $t\geq T$, then $\Omega$ is a selected diamond.

If $P_l$  is not in the region $t\geq T$, then diamonds containing $P_u P_l$ or $P_{u'} P_l$ as north boundary are not selected, so except these diamonds, the only diamond including $P_u P_l$ or $P_{u'} P_l$ is $\Omega$. Hence $P_u P_l$ and $P_{u'} P_l$ are either both selected or both not selected.

\bigskip

%
\paragraph{\bf  Step 2}
Recall that the R/C characteristic does not change along forward and backward characteristic directions.
Also using the claim proved in the step 1, to prove the lemma, we only have to show that, in any diamond, the minimum difference in $x$ for any 
 $\fa R_r$ and $\ba R_r$ south jump edges is less or equal to 
the minimum difference in $x$ for any  $\fa R_r$ and $\ba R_r$ north jump edges.

We discuss case by case for diamonds including $\fa R_r$ or $\ba R_r$.

\paragraph{\bf (2.1)}
We consider a $\fa R_r \ba C$ or a $\fa R_r \ba R_c$ diamond shown in Figure \ref{RCRRo}. We use
subscripts $E$, $W$, $N$ and $S$ to denote functions related to east, west, north and south endpoints of the diamond, respectively. And
we use
subscripts $NE$, $NW$, $SE$ and $SW$ to denote functions related to North-East, North-West, South-East and South-West boundary
jump edges of the diamond, respectively. Especially we use $l_{NE}$, $l_{NW}$, $l_{SE}$ and $l_{SW}$ to denote the North-East, North-West, South-East and South-West boundary
jump edges of the diamond.

We want to show that 
\beq\label{lemma_rc}
x_E-x_N\geq x_S-x_W.
\eeq

	\begin{figure}[htp] \centering
			\includegraphics[width=.28\textwidth]{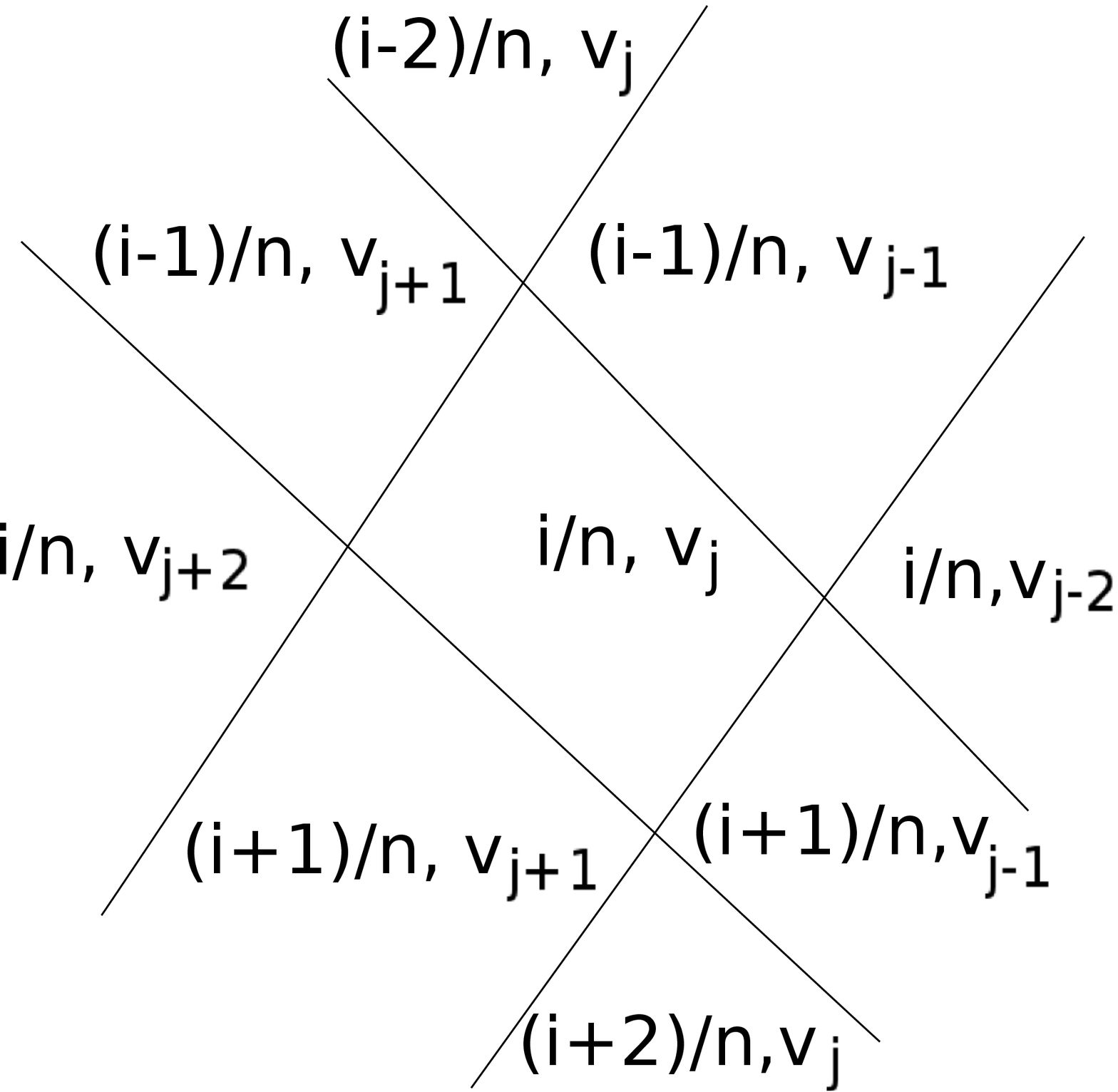}
			\qquad \includegraphics[width=.28\textwidth]{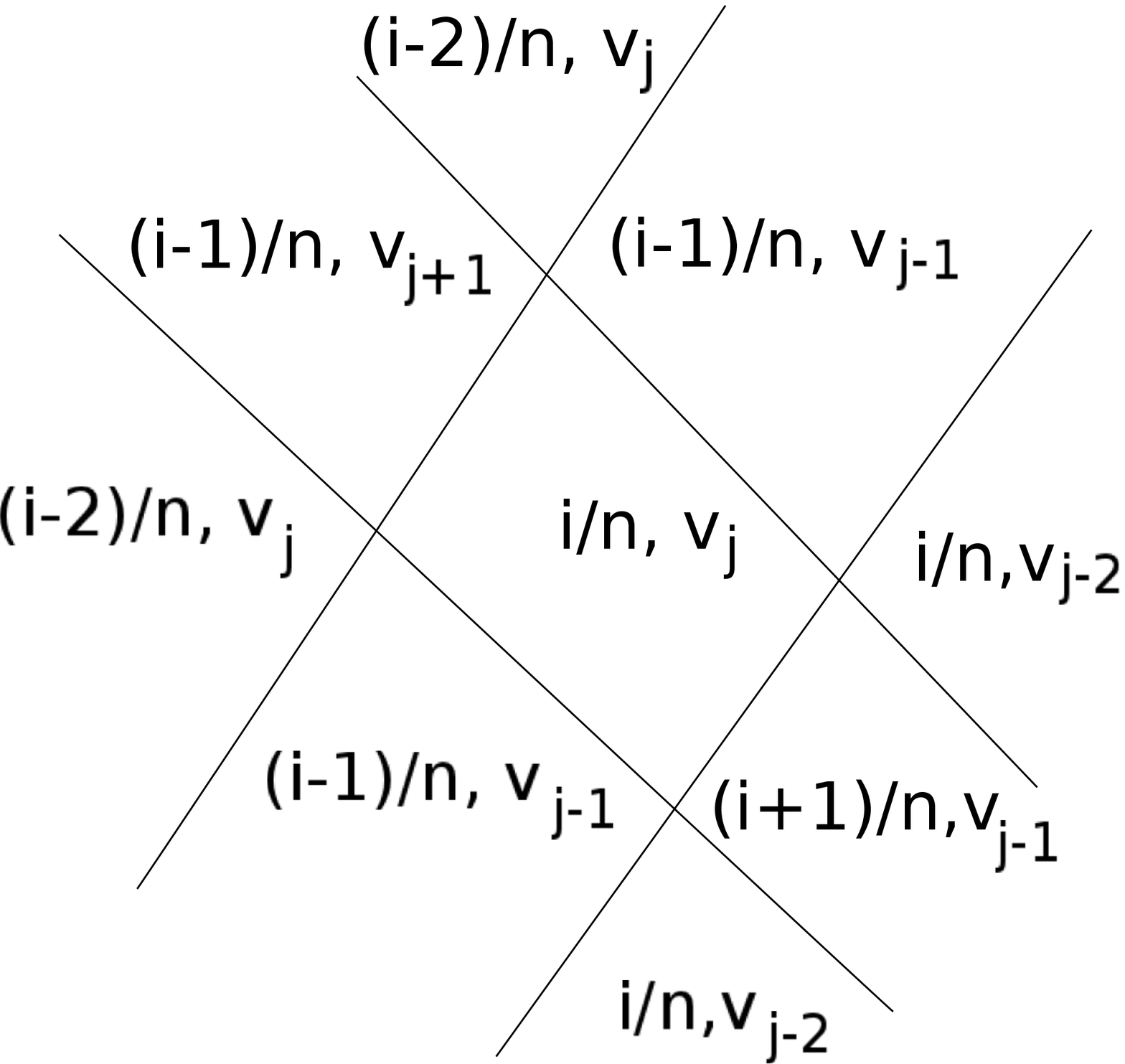}
			\qquad \includegraphics[width=.28\textwidth]{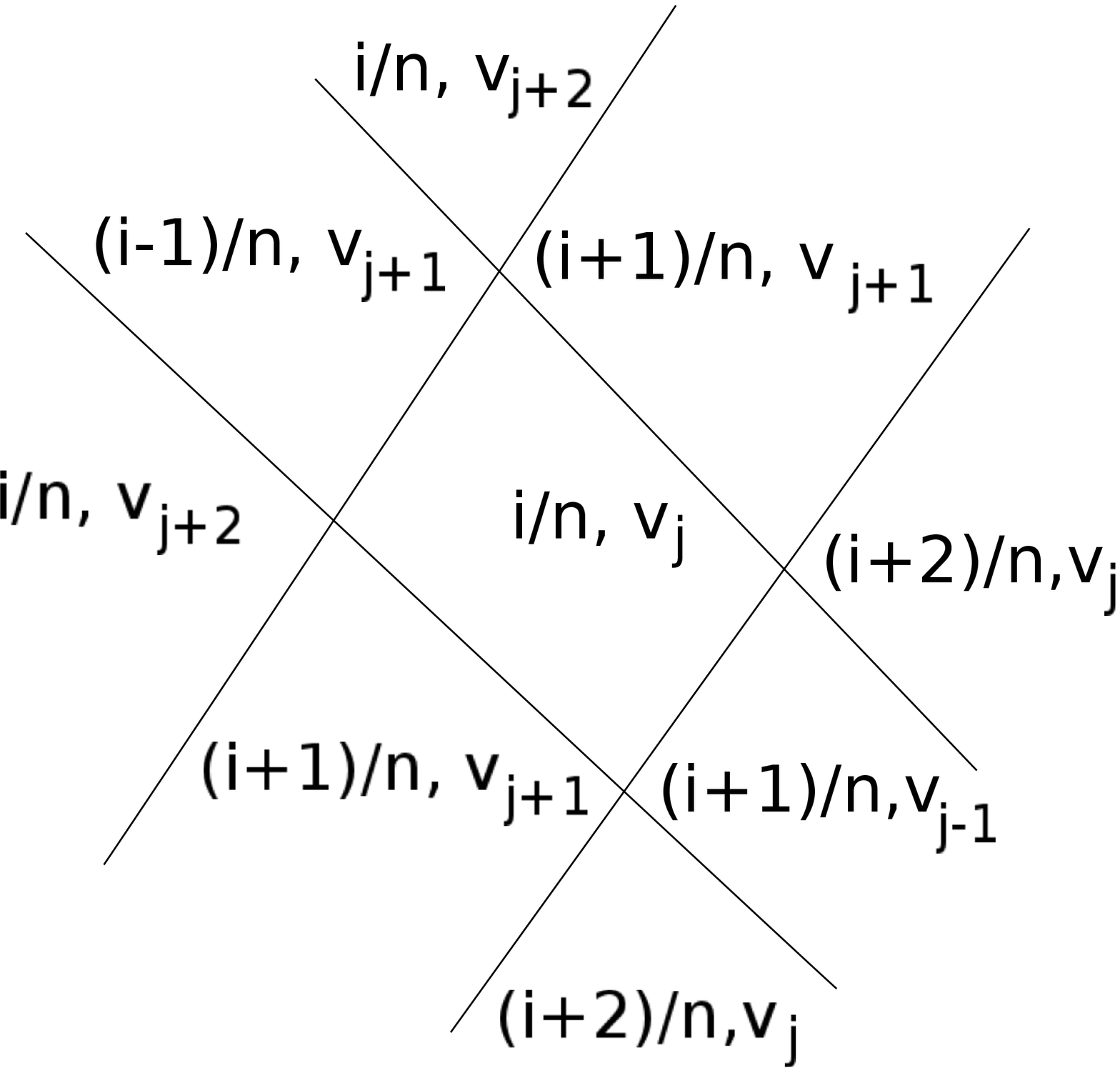}
		\caption{\label{RCRRo}From left to right: the middle diamonds are: $R_rC_c$; $R_r C_r$; and  $R_rR_c$ diamonds, respectively, where forward character always goes first. }
	\end{figure}

By studying the three possible cases in Figure \ref{RCRRo} using 
\eqref{geng2}$\sim$\eqref{geng3}, we always have 
\beq\label{lemma_rc_proof}
0<\frac{\lambda_{NW}}{\lambda_{SE}}<1,\qquad 1\leq\frac{\lambda_{NE}}{\lambda_{SW}},
\eeq
where recall that subscripts $N$, $S$, $W$, $E$ denote the north, south, west and east endpoints of the middle
diamond in the figure, respectively.

Then we could give an easy geometric proof for \eqref{lemma_rc}  in Figure \ref{lemma_proof}.
In fact,  drawing two dash lines parallel to $l_{NW}$ and $l_{SW}$, respectively, then by \eqref{lemma_rc_proof}, 
we know the parallelogram is inside the diamond, although in $\fa R_r \ba C_r$ and  $\fa R_r \ba R_c$ diamonds in Figure \ref{RCRRo} one edge of the parallelogram lies on $l_{NE}$. Then by Figure \ref{lemma_proof},
clearly  \eqref{lemma_rc} is correct.

	\begin{figure}[htp] \centering
			\includegraphics[width=.3\textwidth]{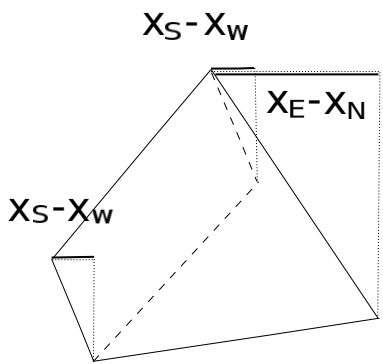}
		\caption{\label{lemma_proof}The proof of  \eqref{lemma_rc} on a diamond satisfying \eqref{lemma_rc_proof}.}
	\end{figure}

\bigskip

\paragraph{\bf (2.2)}
However, the easy geometric proof in the previous part is not correct for $\fa R_r \ba R_r$ interaction where the second
inequality in \eqref{lemma_rc_proof} is in the opposite direction. Instead we use another method to cope with a $\fa R_r \ba R_r$ diamond. 

We use similar idea as the one in \cite{lin2}. 

	\begin{figure}[htp] \centering
			\includegraphics[width=.3\textwidth]{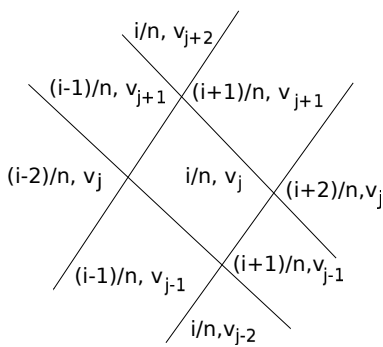}
		\caption{\label{RR}The center diamond is a $R_rR_r$ block.}
	\end{figure}

We use $B$ and $D$ to denote the difference in $x$ and $t$ for each adjacent pair of endpoints on the diamond, respectively, such as 
\beq
B_{NW}=|x_N-x_W|,\qquad D_{NW}=|t_N-t_W|,
\eeq
so clearly we have
$$B_{NW}+B_{NE}=B_{SW}+B_{SE},$$
$$D_{NW}+D_{SW}=D_{NE}+D_{SE},$$
and 
$$
\frac{B_{NW}}{D_{NW}}=\frac{B_{NE}}{D_{NE}}=\lambda_{NW}<\lambda_{SE}
=\frac{B_{SW}}{D_{SW}}=\frac{B_{SE}}{D_{SE}}\,.
$$

Denote that 
\[
0<\alpha=\frac{\lambda_{NW}}{\lambda_{SE}}<1,
\]
then we have 
$$B_{NW}=(\frac{1}{2} +\frac{1}{2}\alpha)B_{SE}+(\frac{1}{2} -\frac{1}{2}\alpha)B_{SW},$$
and 
$$B_{NE}=(\frac{1}{2} +\frac{1}{2}\alpha)B_{SW}+(\frac{1}{2} -\frac{1}{2}\alpha)B_{SE},$$
hence we have
\beq
\min(B_{NW}, B_{NE})>\min(B_{SW}, B_{SE}),
\eeq
which is the estimate we need.

Combining all information we have, we already finished the proof of this lemma.
\end{proof}
\bigskip

Finally we prove Theorem \ref{main0}.
\begin{proof}
We prove Theorem \ref{main0} in two steps.
\paragraph{\bf  Step 1}
First we show that: Suppose the polygonal scheme is well-defined when $0<t\leq T^*$. Then we have,
when $n$ is sufficiently large,
\beq\label{main_esti}
{v}^{(n)}(x,t)\leq \max_x{{v}^{(n)}_0(x)}+L\cdot t\com{when} 0<t\leq T^*
\eeq
for a uniform constant $L$ independent of $n$.

By Lemma \ref{No-R}, we know  that ${v}^{(n)}$ is not increasing along some direction in any districts except
$\fa R_r\ba R_r$ districts.
Hence in these districts clearly we have 
\beq{\label{main_proof_1}}
\max_{\text{in whole district}}{{v}^{(n)}}\leq\max_{\text{on lower boundary of the district}}{{v}^{(n)}}+\frac{1}{n}
\eeq

	\begin{figure}[htp] \centering
			\includegraphics[width=.1\textwidth]{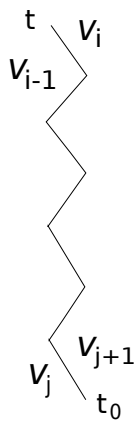}
		\caption{\label{finalproof}Bound on ${v}^{(n)}$ in a $R_rR_r$ district. In the figure, we omit the subscript ${(n)}$ for convenience.}
	\end{figure}

Next, in the $\fa R_r\ba R_r$ district, we will use the result obtained in Lemma \ref{a_t_in} that $a^{(n)}(t)$ is not decreasing on $t$
to prove \eqref{main_esti}. To see it, choose a block with ${v}={v}_i^{(n)}(t)$ inside this block, then trace it back to a 
block with ${v}={v}^{(n)}_j(t_0)$ on the lower boundary of the considered $\fa R_r\ba R_r$ district by a series of jump edge shown in Figure \ref{finalproof}. The ${v}$ values in Figure \ref{finalproof}
are given according to a fact that in the $\fa R_r\ba R_r$ district if ${v}^{(n)}={v}^{(n)}_k$ ahead of a jump edge
then ${v}^{(n)}={v}^{(n)}_{k+1}$ behind that jump edge. 
Then by \eqref{geng2}$\sim$\eqref{geng3} and the definition of $a^{(n)}(t)$, we have 
\[
t-t_0\geq a^{(n)}(0) \sum_{k=j}^i n({v}^{(n)}_{k+1}-{v}_k)=n a^{(n)}(0)({v}^{(n)}_i-{v}^{(n)}_j),
\]
which immediately implies that
\beq\label{main_proof_2}
{v}^{(n)}_i\leq {v}^{(n)}_j+\frac{1}{n a^{(n)}(0)}(t-t_0)\,.
\eeq

Then using Lemma \ref{lemma_a_0}, \eqref{main_proof_1} and \eqref{main_proof_2}, we can prove (\ref{main_esti}), where 
note for each polygonal scheme, there are at most finite many districts.
\bigskip
\paragraph{\bf  Step 2.}
The local-in-time existence result in \cite{lxy} shows that under the assumption in Theorem \ref{main0},
there exists a time interval $t\in[0,\ve]$ in which the polygonal schemes are well-defined when $n$ is sufficiently large
and converging to a Lipschitz continuous solution for \eqref{ps} as $n$ approaches infinity. The constant $\ve$ is only dependent on Lipschitz norms on ${v}$ and $u$, but independent of $n$.

By the weak-strong uniqueness of the classical solution for \eqref{ps}, c.f. \cite{Dafermos2010},
we know the Lipschitz continuous solution of  \eqref{ps} when $t\in[0,\ve]$ agrees with the solution through the limit of polygonal scheme. Hence any Lipschitz continuous solution satisfies
\eqref{main0_esti_0} by the a priori estimate in (\ref{main_esti}) for the approximation solution.

Then, for any finite time $T>0$ and any Lipschitz continuous solutions on $[0,T]$, repeat
above process finite many times, we prove that the solution always satisfies \eqref{main0_esti_0} when $t\in[0,T]$,
where in each time we could evolve by a time step $\ve$ which is constant. This complete the proof of the theorem.
\end{proof}
\bigskip

Finally, we apply Theorem \ref{main0} to achieve a better estimate for the life-span of classical solutions including compression than \cite{CPZ}, when $1<\gamma<3$. Before stating the corollary, we first review a lemma coming from \cites{lax2,CPZ}.
\begin{lemma} \cites{lax2,CPZ} \label{p_lemma_1}
For $C^1$ solutions of (\ref{ps}) we have
\begin{align} 
  \partial_+ y &= - K_0
        {v}^{\frac{\gamma-3}{4}} \, y^2\,, \label{p_y_eq}\\
  \partial_- q &= - K_0
        {v}^{\frac{\gamma-3}{4}}\, q^2\,, \label{p_q_eq}
  \end{align}
where
\[
  y(x,t) :=  \sqrt{c}\,s_x, 
\qquad
  q(x,t):=    \sqrt{c}\,r_x
\]
and $K_0$ is a constant only depending on $\gamma$ which can be easily found in \cite{CPZ}.
\end{lemma}
\begin{corollary}\label{cor_main}
Assume all assumptions in Theorem \ref{main0} hold, the initial data $s_0(x)$ and $r_0(x)$ are $C^1$, and 
\beq\label{y_q_neg}
G_0:=\min_x\big(y(x,0),q(x,0)\big)<0,
\eeq
i.e. the initial data are compressive somewhere, then singularity happens not later than
\[
t=\frac{1}{L}\left\{\Big(-\frac{4K_0}{\gamma+1}\frac{1}{G_0}+H_0^{\frac{\gamma+1}{4}}\Big)^{\frac{4}{\gamma+1}} -H_0\right\}\,,
\]
where we denote 
\[
H_0:=\max_x\big({v}(x,0)\big)\,.
\]
\end{corollary}
\begin{proof}
Without loss of generality, we assume that 
\[
y(x^*,0)=G_0+\ve=\min_x\big(y(x,0),q(x,0)\big)+\ve<0,
\]
where $0<\ve\ll 1$ is a constant.

For smooth solution, along a forward characteristic $x^+(t)$ starting from $(x^*,0)$, by \eq{p_y_eq}, we have
\[
\frac{1}{y\big(x^+(t),t\big)}=\frac{1}{y\big(x^*,0\big)}+\int_{0}^{t}K_0
        {v}^{\frac{\gamma-3}{4}}\big(x^+(\xi),\xi\big) d\xi\,.
\]
Then right hand side of this equation equals to zero, i.e. $y\big(x^+(t),t\big)$ blows up, not later than
a time $t$ satisfying
\[
-\frac{1}{y\big(x^*,0\big)}=\int_{0}^{t}K_0
        {v}^{\frac{\gamma-3}{4}}\big(x^+(\xi),\xi\big) d\xi\,.
\] 
Then by \eqref{main0_esti_0}  in Theorem \ref{main0}, it is very easy to prove this corollary,
where we use that $\ve$ can be arbitrarily small.
\end{proof}
\section*{Acknowledgments} We appreciate the helpful discussion with Professor Helge
Kristian Jenssen.
Shenguo Zhu is supported in part
by National Natural Science Foundation of China under grant 11231006, Natural Science Foundation of Shanghai under grant 14ZR1423100 and  China Scholarship Council.

\end{document}